%% file: _AiML2020_Main.tex
\documentclass[twoside]{aiml20}

\usepackage{aiml20macro}

\usepackage{amsmath}
\usepackage{amssymb}
\usepackage{xspace}

\input{9_macros.tex}

\def\lastname{Kikot, Shapirovsky and Zolin}

\begin{document}

\begin{frontmatter}

\title{Completeness of logics with the transitive closure modality and related logics}
 \author{Stanislav Kikot}\footnote{\texttt{staskikotx@gmail.com}}
 \address{University of Oxford
}
 \author{Ilya Shapirovsky}\footnote{\texttt{ilya.shapirovsky@gmail.com}.
 The work on this paper was supported by the Russian Science Foundation under grant 16-11-10252 and carried out at Steklov Mathematical Institute of Russian Academy of Sciences.}
 \address{
 Steklov Mathematical Institute of Russian Academy of Sciences 
 and 
Institute for Information Transmission Problems of Russian Academy of Sciences}
 \author{Evgeny Zolin}\footnote{\texttt{ezolin@gmail.com}.
 The article was prepared within the framework of the HSE University Basic Research Program.}
 \address{National Research University Higher School of Economics, Russian Federation\\
}
  
\begin{abstract} 
We give a sufficient condition for Kripke completeness 
of modal logics enriched with the transitive closure modality.
More precisely, we show that if a logic admits
what we call definable filtration (ADF),
then such an expansion of the logic is complete;
in addition, has the finite model property, and again ADF.
This argument can be iterated, and as an application we obtain
the finite model property for PDL-like expansions of logics that ADF.
\end{abstract}


\end{frontmatter}


\input{0_Introduction.tex}
\input{1_Preliminary.tex}
\input{2.0_Filtration.tex}
\input{2.1_Filter_frames.tex}
\input{2.2_Filter_models.tex}
\input{3_Complete_TrCl.tex}

\input{4_Locally_finite.tex}
\input{5_MFP_formulas.tex}

\input{7_Open_conslusion.tex}


\input{references.tex}
\newpage

\input{8_Appendix.tex}

\end{document}

%% file: 9_macros.tex
\usepackage{mns}

\newcommand\hide[1]{}

\newcommand\ISH[1]{
    {\color{blue}{\bf [}I.Sh.:}    {\it #1}{\color{blue}{\bf]}}
}
\renewcommand\ISH[1]\empty

\def\MFP{\mathrm{MFP}}

\newcommand{\Models}{\mathop{\mathsf{Mod}}}
\newcommand{\Frames}{\mathop{\mathsf{Fr}}}
\newcommand{\Th}{\mathop{\mathsf{Th}}} 

\newcommand{\mL}{\Models(\bL)}
\newcommand{\fL}{\Frames(\bL)}

\def\mo{\models}
\def\vf{\varphi}

\newcommand{\eg}{e.g.\@\xspace}
\newcommand{\ie}{i.e.\@\xspace}
\newcommand{\cf}{cf.\@\xspace}

\newcommand{\wrt}{w.r.t.\@\xspace}

\newcommand{\compact}{\itemsep=0pt \parsep=0pt \parskip=0pt}

\newcommand{\AND}{\textstyle{\bigwedge}}

\newcommand{\B}{\ensuremath{{\square}}\xspace}
\newcommand{\D}{\ensuremath{{\Diamond}}\xspace}

\newcommand{\BP}{{\boxplus}}
\newcommand{\DP}{\mbox{\raisebox{0.2ex}{${\diamondtimes}$}}}

\newcommand{\Fi}{\varphi}
\newcommand{\Var}{\mathsf{Var}}
\newcommand{\Fm}{\mathsf{Fm}}
\newcommand{\Sub}{\ensuremath{\mathop{\mathsf{Sub}}}\xspace}
\newcommand{\se}{\subseteq}
\newcommand{\su}{\supseteq}
\newcommand{\defeq}{\leftrightharpoons}
\newcommand{\nmodels}{\mathrel{\,\not\!\models}}
\newcommand{\valid}{\models}

\newcommand{\proves}{\vdash}
\newcommand{\nproves}{\nvdash}
\newcommand{\ff}[1]{\widehat{#1}}

\newcommand{\fa}{\ff{a}}
\newcommand{\fb}{\smash{\ff{b}}}
\newcommand{\fx}{\ff{x}}
\newcommand{\fy}{\ff{y}}
\newcommand{\fz}{\ff{z}}
\newcommand{\fW}{\smash{\ff{W}}}
\newcommand{\fR}{\smash{\ff{R}}}
\newcommand{\fS}{\smash{\ff{S}}}
\newcommand{\fV}{\smash{\ff{V}}}
\newcommand{\fF}{\smash{\ff{F}}}
\newcommand{\fM}{\smash{\ff{M}}}

\newcommand{\lra}{\leftrightarrow}
\newcommand{\Iff}{\Leftrightarrow}
\newcommand{\IFF}{\Longleftrightarrow}
\newcommand{\To}{\Rightarrow}

\newcommand{\TO}{\Longrightarrow}
\newcommand{\OT}{\Longleftarrow}

\newcommand{\bL}{{L}}
\newcommand{\bK}{\mathbf{K}}

\newcommand{\vL}{\bL}

\newcommand{\cF}{\mathcal{F}}
\newcommand{\cM}{\mathcal{M}}

\newcommand{\BX}[1]{\!\mathop{\texttt{\upshape[}#1\texttt{\upshape]}\!}}
\newcommand{\DX}[1]{\!\mathop{\texttt{\upshape$\langle$}#1\texttt{\upshape$\rangle$}\!}}

\newcommand{\Min}[2][]{R^{\mathsf{min}}_{{#2},{#1}}}
\newcommand{\MinR}[1]{R^{\mathsf{min}}_{{#1}}}

\newcommand{\Max}[2][]{R^{\mathsf{max}}_{{#2},{#1}}}
\newcommand{\MaxR}[1]{R^{\mathsf{max}}_{{#1}}}

\newcommand{\too}{\longrightarrow}
\newcommand{\tooo}[1]{\stackrel{#1}{\too}}

\newcommand{\PDL}{\ensuremath{\mathsf{PDL}}\xspace}


%% file: 0_Introduction.tex
\section*{Introduction}

This paper makes a contribution to the study
of modal logics enriched by the transitive closure modality.

Modal logics that, in addition to the modal operator $\B$ for a binary relation~$R$, also contain the operator $\BP$ 
for the transitive closure of~$R$, are quite common. For instance,
such are operators `everyone knows that' and the operator 
of common knowledge in epistemic logic \cite{fagin2003reasoning}. Other examples include
the logic $\mathsf{Team}$ for collective beliefs and actions \cite{dunin2011teamwork}, 
$\mathrm{Log}(\mathbb{N}, \mathsf{succ}, <)$ and the proposition dynamic logic (PDL) \cite{fischer1979propositional}.

So far, completeness theorems and decidability procedures for such logics have
had bespoke proofs, though many of them refer to Segerberg's \cite{Segerberg1977} 
and Kozen and Parikh's arguments for PDL \cite{KozenParikh1981}.

In this paper we present a toolkit for proving results on completeness, 
finite model property and decidability for such logics in a general setting.
In Section~\ref{Sect:Complete:TransCl} we suggest a sufficient condition for a 
unimodal logic $L$ that ensures 
that axioms for $L$ together with Segerberg's axioms for the transitive 
closure modality are complete for the bimodal logic obtained by enriching the frames for 
$L$ with the transitive closure of their accessibility relation.
In order to state the condition,
we come up with a hierarchy of `admits filtration' properties
(thus extending our earlier work~\cite{KSZ:AiML:2014}) presented in Section~\ref{Sect:Filtration}.
In Section~\ref{Sect:pdlization} we give examples of logics that satisfy these properties
and show how our sufficient condition can be iterated for obtaining completeness
for `PDLizations' of a family of logics. 
As a by-product we obtain the finite model property
and often decidability (through Harrop's theorem) for the logics in question.

%% file: 1_Preliminary.tex
\section{Preliminaries}\label{Sect:Prelim}

We assume the reader to be familiar with syntax and semantics of multi-modal logic
\cite{B:R:V:ML:2002,Ch:Za:ML:1997},
so we only briefly recall some notions and fix notation.
Let $\Sigma$ be a finite alphabet (of indices for modalities).
The set $\Fm(\Sigma)$ of \emph{modal formulas} (\emph{over~$\Sigma$}) is defined from
propositional letters $\Var=\{p_0,p_1,\ldots\}$
using Boolean connectives and the modalities $\BX{e}$, for ${e\in\Sigma}$,
according to the syntax:
\begin{center}$
\Fi \ ::=\ \bot\ \mid\ p_i\ \mid \ \Fi\to\psi \ \mid\ \BX{e}\Fi.
$\end{center}
We use standard abbreviations (\eg, $\top,\land$);
in particular, $\DX{e}\Fi:=\neg\BX{e}\neg\Fi$.
For a set of formulas $\Gamma$,
by $\Sub(\Gamma)$ we denote the set of all subformulas of formulas from~$\Gamma$.
We say that $\Gamma$ is \emph{$\Sub$-closed} if ${\Sub(\Gamma)\se\Gamma}$.

A ($\Sigma$-)\emph{frame} is a pair $F=(W,(R_e)_{e\in\Sigma})$,
where ${W\ne\varnothing}$ and ${R_e\se W{\times}W}$ for ${e\in\Sigma}$.
A \emph{model} based on~$F$ is a pair $M=(F,V)$, where ${V(p)\se W}$, for all ${p\in\Var}$.
The \emph{truth relation} ${M,x\models\Fi}$ is defined in the usual way, \eg
\begin{center}$
M,x\models\BX{e}\Fi \quad \defeq \quad \mbox{for all $y\in W$, \ if $x\,R_e\,y$ \ then \ $M,y\models\Fi$.}
$\end{center}
We write $M\mo\vf$ if $M,x\mo\vf$ for all $x$ in $M$.   
A formula $\Fi$ is \emph{valid} on~$F$, notation ${F\models\Fi}$,
if $M\models\Fi$ for all $M$ based on~$F$. 
For a class of frames~$\cF$, an \emph{$\cF$-model} is a model based on a frame from~$\cF$.

A (\emph{normal modal}) \emph{logic} (\emph{over $\Sigma$}) is a set of formulas $\bL$
that contains 
all classical tautologies, the axioms $\BX{e}({p\to q})\to(\BX{e}p\to\BX{e}q)$, for each ${e\in\Sigma}$,
and is closed under the rules of
modus ponens, 
substitution, 
and necessitation (from $\Fi$, infer $\BX{e}\Fi$, for each ${e\in\Sigma}$).
An \emph{$\bL$-frame} is a frame $F$ such that ${F\valid L}$. 
\emph{The logic of a class of frames}~$\cF$ is the set
of all formulas that are valid in~$\cF$.
A~logic $L$ is \emph{Kripke complete} if it is the logic of some class of frames.
A~logic $\bL$ has the \emph{finite model property} (FMP)
if it is the logic of some class of finite frames; or equivalently (see, e.g., \cite[Th.~3.28]{B:R:V:ML:2002})
if, for every formula ${\Fi\notin\bL}$, there is a finite model $M$ 
such that $M\mo L$ and ${M\nmodels\Fi}$.
%
%
%
For a logic $\vL$, put
\begin{center}
$\begin{array}{rcl@{\ }l@{\ }r}
\fL & = & \{\, F\,\mid F & \mbox{is a frame and} & F \models\vL\},\\
\mL & = & \{\, M\mid M & \mbox{is a model and} & M \valid \vL\}.
\end{array}$
\end{center}

Clearly, every $\Frames(\vL)$-model belongs to $\Models(\vL)$.
The converse does not hold in general; \eg, the canonical model of a non-canonical logic
$\vL$ is not a $\Frames(\vL)$-model.
But the converse holds in the following special case.
A~model $M$ is called \emph{differentiated} if
any two points in~$M$ can be distinguished by a formula.

\begin{lemma}[{\rm See e.g. \cite[Ex.~4.9]{GoldbLTC}}]\label{Lem:Fin:Diff:Model}
Let $M=(F,V)$ be a finite differentiated model. If all substitution
instances of a formula $\Fi$ are true in~$M$, then ${F\valid\Fi}$.

In particular, if ${M\models L}$, where $L$ is a logic, then ${F\valid L}$.
\end{lemma}

\noindent{\bf Harrop's theorem.}
{\it A finitely axiomatizable logic with the FMP is decidable.}

%% file: 2.0_Filtration.tex
\section{Filtration}\label{Sect:Filtration}

The notion of a filtration we introduce below slightly generalizes
the standard one (\cf~\cite[Def.\,2.36]{B:R:V:ML:2002}, \cite[Sect.\,5.3]{Ch:Za:ML:1997})
in the following aspect:
given a finite set of formulas $\Gamma$,
we define a filtration as a model obtained by
factoring a given model through an equivalence relation
that we allow to be \emph{finer} than the one induced by~$\Gamma$.
This modification seems to first appear in~\cite{Shehtman1987};
see also~\cite{Shehtman:AiML:2004}.

Let $M=(W,(R_e)_{e\in\Sigma},V)$ be a model and $\Gamma$ a finite $\Sub$-closed set of $\Sigma$-formulas.
An equivalence relation~$\sim$ on~$W$ is \emph{of finite index}
if the quotient set $W/{\sim}$ is finite.
The equivalence relation \emph{induced by~$\Gamma$} is defined as follows:
\begin{center}$
x\,\sim_\Gamma\, y \qquad\defeq\qquad
\forall\Fi\in\Gamma\ \bigl(\, M,x\models\Fi\ \Iff\ M,y\models\Fi \,\bigr).
$\end{center}
Clearly, $\sim_\Gamma$ is of finite index.
We say that an equivalence relation $\sim$ \emph{respects}~$\Gamma$ if
${\sim}{}\subseteq{}{\sim_{\Gamma}}$;
in other words, if for every class~${[x]_\sim\se W}$ and every formula ${\Fi\in\Gamma}$,
$\Fi$ is either true in all points of~$[x]_\sim$ or false in all points of~$[x]_\sim$.

\begin{definition}{\bf (Filtration)}\label{Def:Filtration}
By a \emph{filtration} of a model $M$ that \emph{respects} a set of formulas~$\Gamma$
(or a \emph{$\Gamma$-filtration of~$M$})
we call any model $\ff{M}=(\ff{W},(\ff{R}_e)_{e\in\Sigma},\ff{V})$
that satisfies the following conditions:
\begin{itemize}
\item
$\ff{W}=W/{}{\sim}$, for some equivalence relation of finite index $\sim$ on~$W$;
\item
 the equivalence relation $\sim$ respects~$\Gamma$, \ie, 
$x\sim y$ implies $x\sim_\Gamma y$;
\item
the valuation $\ff{V}$ is defined on the variables $p\in\Gamma$ canonically:
${\ff{x}\models p}$ $\Iff$ ${x\models p}$,
for all points ${x\in W}$, where $\ff{x}:=[x]_\sim$ denotes the $\sim$-class of a point~$x$;
\item
$\Min[e]{\sim}\se\ff{R}_e\se\Max[e]{\Gamma}$, for each ${e\in\Sigma}$.
Here $\Min[e]{\sim}$ is the $e$-th \emph{minimal filtered relation} on~$\ff{W}$, 
and $\Max[e]{\Gamma}$ is the $e$-th \emph{maximal filtered relation}
on~$\ff{W}$
induced by the set of formulas~$\Gamma$; 
they are defined in the usual way:
$$
\begin{array}{r@{\quad}c@{\quad}l}
\ff{x}\ \Min[e]{\sim}\ \ff{y} & \defeq & \exists x'\sim x\ \, \exists y'\sim y\colon\ \,x'\,R_e\, y',\\
\ff{x}\ \Max[e]{\Gamma}\ \ff{y} & \defeq &
\mbox{for every formula $\BX{e}\Fi\in\Gamma$}
\ \bigl(\,M,x\models\BX{e}\Fi\ \To\ M,y\models\Fi\,\bigr).
\end{array}
$$
\end{itemize}
If $\sim\,=\,\sim_\Phi$ for some finite set of formulas~$\Phi$, then we call
$\ff{M}$ a \emph{definable $\Gamma$-filtration} of~$M$ (\emph{through~$\Phi$}).
\end{definition}

Note that the relations $\Min[e]{\sim}$ and $\Max[e]{\Gamma}$
are well-defined and ${\Min[e]{\sim}\se\Max[e]{\Gamma}}$.
The condition ${\Min[e]{\sim}\se\ff{R}_e}$ is equivalent to that
${\forall x,y\in W}$ ($x\,R_e\,y$ $\To$ $\ff{x}\,\ff{R}_e\,\ff{y}$).
A~filtration is always a finite model.
If ${\sim\,=\,\sim_\Phi\,\se\,\sim_\Gamma}$, then w.l.o.g.\ ${\Phi\su\Gamma}$.
The following is the key lemma about filtration
(\cf~\cite[Th.\,2.39]{B:R:V:ML:2002},
\cite[Th.~5.23]{Ch:Za:ML:1997}).

\begin{lemma}[Filtration lemma]\label{Lemma:Filtration}
Let $\Gamma$ be a finite $\Sub$-closed set of formulas and~$\ff{M}$
is a $\Gamma$-filtration of a model~$M$.
Then, for all points ${x\in W}$ and all formulas ${\Fi\in\Gamma}$,
we have: \
${M,x\models\Fi}$ $\Iff$ ${\ff{M},\ff{x}\models\Fi}$.
\end{lemma}

%% file: 2.1_Filter_frames.tex
\subsection{Admissibility of filtration}\label{Sect:Filtr:Frames}

\begin{definition}{\bf (ADF for classes of frames)}\label{Def:Admit:Filtration:Frames}
We say that a class of frames $\cF$ \emph{admits (definable) filtration}
if, for any finite \Sub-closed set of formulas~$\Gamma$
and an $\cF$-model $M$,
there exists an $\cF$-model that is a (definable) $\Gamma$-filtration of~$M$.
\end{definition}


\begin{lemma}[AF for frames implies FMP]\label{Lem:AF:Fr:FMP}\ \\
If a logic $\bL$ is Kripke complete and $\Frames(\bL)$
admits filtration, then $\bL$ has the FMP.
\end{lemma}

%% file: 2.2_Filter_models.tex

\begin{definition}{\bf (ADF for classes of models)}\label{Def:Admit:Filtration:Models}
We say that a class of models $\cM$ \emph{admits  (definable) filtration}
if, for any finite \Sub-closed set of formulas~$\Gamma$
and a model ${M{\in}\cM}$,
there a model in $\cM$ that is a (definable) $\Gamma$-filtration of~$M$.

\end{definition}


\begin{lemma}[AF for models implies FMP]\label{Lem:AF:Mo:FMP} \ \\
%
If $\mL$ admits filtration, then the logic $\bL$ has the FMP and hence is complete.
\end{lemma}

Note that this lemma does not involve the completeness assumption,
since any normal logic $\bL$ is complete \wrt the class of its models $\Models(\bL)$.

We have two variants of the notion ``a logic $\bL$ admits (definable) filtration'':
\\[0.5ex](I) \ the class of frames $\fL$ admits (definable) filtration (Definition~\ref{Def:Admit:Filtration:Frames});
\\(II) \ the class of models $\mL$ admits (definable) filtration (Definition~\ref{Def:Admit:Filtration:Models}).

\smallskip
In both variants, we filtrate a model $M=(F,V)$ into a model $\fM=(\fF,\fV)$.
The precondition (${F\valid\bL}$) in (I) is stronger than that (${M\models\bL}$) in~(II).
The postcondition (${\fF\valid\bL}$) in (I) is stronger than (${\fM\models\bL}$) in~(II), too.
However, we can always make sure that the finite model $\fM$ is differentiated.
Then ${\fM\models\bL}$ iff ${\fF\models\bL}$.
Thus, (II) is stronger than~(I), as the following lemma states.

\begin{lemma}[ADF for models implies ADF for frames]\label{Lem:ADF:Mod:Fr}\ \\
For any logic $\bL$, if $\mL$ admits (definable) filtration, then so does $\fL$.
\end{lemma}
\begin{proof}
Take any finite $\Sub$-closed set of formulas $\Gamma$
and a model $M=(F,V)$ with ${F\valid\bL}$. Then 
${M\in\mL}$. Since $\mL$ A(D)F,
the model $M$ has a (definable) $\Gamma$-filtration $\fM=(\fF,\fV)$
with ${\fM\models\bL}$.
The model $\fM$ is finite and, 
without loss of generality, differentiated,
by Lemma~\ref{Lem:Differentiated} (in Appendix).
Then ${\smash{\fF}\valid\bL}$, by Lemma~\ref{Lem:Fin:Diff:Model}.
Thus, $\fL$ admits (definable) filtration.
\end{proof}

Next we prove that, for the \emph{canonical} logics, the notions (I) and (II) coincide,
if we consider \emph{definable} filtration.
To simplify notation, we work with the unimodal case.
Recall that one can build
the \emph{canonical} frame $F_T=(W_T,R_T)$
and model $M_T=(F_T,V_T)$
not only for a (consistent) normal \emph{logic},
but more generally for a \emph{normal theory}~$T$ (which contains all theorems of $\bK$
and is closed under monus ponens and necessitation).
Any point ${x\in W_T}$ is a
consistent (never ${A,\neg A\in x}$) complete (always ${A\in x}$ or ${\neg A\in x}$)
theory (\ie, closed under modus ponens) containing~$T$.
A~logic $L$ is called \emph{canonical} if ${F_L\valid L}$.
The following well-known fact is useful.

\begin{lemma}[Canonical generated submodel]
\label{Lem:Canon:Gen:Sub}
If ${T\se T'}$ are consistent normal theories,
then
$M_{T'}$ is a generated submodel of 
$M_T$.
Similarly for frames.
\end{lemma}
\begin{proof}
Assume $x\in W_{T'}$, $y\in{W_T}$, and $x\,R_T y$. To prove that ${y\in W_{T'}}$,
\ie, ${T'\se y}$, take any formula ${A\in T'}$. By normality ${\B A\in T'}$.
Since ${T'\se x}$, we have ${\B A\in x}$. By definition of $R_T$, we obtain ${A\in y}$.
\end{proof}

A typical example of a normal theory is the theory of a model $T=\Th(M)$.
For a model $M=(W,R,V)$, consider the canonical model $M_T$ of its theory
and the \emph{canonical mapping} $t$ from $M$ to $M_T$ defined, for ${a\in W}$, by
\begin{center}
$t(a)=\Th(M,a)\in W_T$.
\end{center}
It is monotonic ($a\,R\,b$ $\To$ $t(a)\,R_T\,t(b)$), but in general it is
neither surjective, nor a p-morphism.
The next lemma shows what happens to the canonical mapping if we filtrate 
both $M$ and $M_T$ through a finite set of formulas~$\Phi$.

\begin{lemma}
\label{Lem:Filtr:canon:map}
Under the above conditions,
any finite set of formulas $\Phi$ induces a bijection
between the 
quotient sets
$W{/}{{\sim}_\Phi}$ and $W_T{/}{{\sim}_\Phi}$ defined, for ${a\in W}\!$,~by
\begin{center}
$f([a]_{\sim_\Phi})\ :=\ [t(a)]_{\sim_\Phi}$.
\end{center}
\end{lemma}
\begin{proof}
This technical proof is put in Appendix, see Lemma~\ref{Lem:Filtr:canon:map:App}.
\end{proof}

\begin{theorem}[ADF for frames implies ADF for models]\label{Th:ADF:Fr:Mod}
If $L$ is a ca\-no\-nical logic $L$, then $\fL$ admits definable filtration iff so does $\mL$.
\end{theorem}
\begin{proof}
($\Leftarrow$) By Lemma~\ref{Lem:ADF:Mod:Fr}.
($\To$) {\sc Idea:} in order to filtrate a model ${M\models L}$,
we filtrate the canonical model $M_T$ of its theory $T=\Th(M)$ and then use the bijection
from Lemma~\ref{Lem:Filtr:canon:map} to transfer the filtration back to~$M$.

Take a finite $\Sub$-closed set of formulas~$\Gamma$ and
a model $M=(W,R,V)$ with ${M\models L}$.
Its theory $T=\Th(M)$ contains~$L$,
hence $F_T$ is a generated subframe of~$F_L$, by Lemma~\ref{Lem:Canon:Gen:Sub}.
Since $L$ is canonical,  we have ${F_L\valid L}$ and so ${F_T\valid L}$.
Thus, $M_T$ is a $\fL$-model and, by assumption, we can filtrate it.

Therefore, the model $M_T$ has a $\Gamma$-filtration $\ff{M_T}=(\ff{W_T},\ff{R_T},\ff{V_T})$
(through some finite set of formulas ${\Phi\su\Gamma}$)
with ${\ff{F_T}\valid L}$.
By Lemma~\ref{Lem:Filtr:canon:map}, there is a
bijection $f$ between the finite sets $\fW=(W/{\sim_\Phi})$ and $\ff{W_T}=(W_T/{\sim_\Phi})$.
Now we build a model $\ff{M}=(\ff{W},\ff{R},\ff{V})$ isomorphic to~$\ff{M_T}$,
by putting, for all ${a,b\in W}$:
\begin{center}
$\ff{a}\,\fR\,\ff{b}$ iff $f(\ff{a})\,\ff{R_T}\,f(\ff{b})$; \qquad
${\ff{a}\models p}$ iff $f(\ff{a})\models p$, for all variables ${p\in\Gamma}$.
\end{center}

%
Since the frames $\ff{F}$ and $\ff{F_T}$ are isomorphic
and ${\ff{F_T}\valid L}$, we have ${\fF\valid L}$.
It remains to prove that $\fM$ is a $\Gamma$-filtration (through~$\Phi$) of~$M$.
Below, we denote $x=t(a)=\Th(M,a)$ and $y=t(b)=\Th(M,b)$,
so that ${f(\fa)=\fx}$ and ${f(\fb)=\fy}$.

{\sf (var)} Let us check that
${\ff{M},\ff{a}\models p}$ iff ${M,a\models p}$, \ for all ${p\in\Gamma}$.
We have:
\begin{center}
       $\fM,\fa\models p$
\ \ $\Iff$ \ \ $\ff{M_T},\fx\models p$
\ \ $\Iff$ \ \ $M_T,x\models p$
\ \ $\Iff$ \ \ $p\in x$
\ \ $\Iff$ \ \ $M,a\models p$.
\end{center}

{\sf (min)} We check that $\MinR{\sim_\Phi}\,\se\,\fR$, \ie,
$\forall{a,b\in W}$ ($a\,R\,b$ $\To$ $\fa\,\fR\,\fb$).

We use the monotonicity of~$t(\cdot)$ and the condition {\sf(min)} for $\ff{R_T}$:
\begin{center}
$a\,R\,b$
\ \ $\TO$ \ \ $t(a)\,R_T\,t(b)$
\ \ $\Iff$ \ \ $x\,R_T\,y$
\ \ $\TO$ \ \ $\fx\,\ff{R_T}\,\fy$
\ \ $\Iff$ \ \ $\fa\,\fR\,\fb$.
\end{center}

{\sf (max)} We check that $\fR\,\se\,\MaxR{\Gamma}$.
Assume $\fa\,\fR\,\fb$. Then $\fx\,\ff{R_T}\,\fy$.

By the condition {\sf(max)} for $\ff{R_T}$,
we have $\fx\,((R_T)^{\sf{max}}_\Gamma)\,\fy$.

We need to show that $\fa\,\MaxR{\Gamma}\,\fb$.
For any formula ${\B A\in\Gamma}$, we have:
\begin{center}
$M,a\models\B A$ $\Iff$ $\B A\in x$ $\Iff$ $M_T,x\models\B A$
$\To$ $M_T,y\models A$ $\Iff$ $A\in y$ $\Iff$ $M,b\models A$.
\end{center}
This completes the proof of the theorem.
\end{proof}

%% file: 3_Complete_TrCl.tex
\section{Logics with the transitive closure modality}\label{Sect:Complete:TransCl}

In this section, $L\se\Fm(\B)$ is a normal unimodal logic. 
Let $L^\BP\se\Fm(\B,\BP)$ be the minimal normal logic 
that extends $L$ with the following axioms 
describing
the interaction
between the modality $\B$ and the \emph{transitive closure} modality $\BP$:
%
\begin{center}$
\mathsf{(A1)}\ \BP p \to \B p,
\qquad
\mathsf{(A2)}\ \BP p \to \B\BP p,
\qquad
\mathsf{(A3)}\ \BP(p\to\B p)\to (\B p \to \BP p).
$\end{center}

Segerberg~\cite{Segerberg1977} (see also~\cite{Segerberg1982,Segerberg:1993:PDL})
and later Kozen and Parikh~\cite{KozenParikh1981}
proved that the logic $\bK^\BP$ (and even PDL) is complete and 
has the FMP; in other words, it is the logic of the class of
finite frames of the form $(W,R,R^+)$; hence it is decidable (more exactly,
{\sc ExpTime}-complete).
Note that the logic $\bK^\BP$ is not canonical:
indeed, canonical logics are compact;
however, the set of formulas $\Gamma=\{\B^n p\mid n\ge1\}\cup\{\neg\BP p\}$
is unsatisfiable (in the class of $\bK^\BP$-frames),
although each finite subset of~$\Gamma$ is satisfiable.
So even for simple logics we cannot use
canonical models as a method of obtaining completeness.



To the best of our knowledge,
up to now, there were no general results on the completeness and decidability
for the $\BP$-companions of logics other than~$\bK$.
Here we obtain one such result.
We give a condition on~$L$ sufficient for the completeness of~$L^\BP$.
The condition is strong enough and guarantees not only the completeness,
but the FMP 
of~$L^\BP$; this limits the scope of our approach.

For simplicity, in this section we assume that $L$ is unimodal.
The results transfer easily to multi-modal logics.
Given a unimodal frame $F=(W,R)$, we denote 
$F^{\oplus}=(W,R,R^+)$.
Given a class of unimodal frames~$\cF$, we denote
$\cF^{\oplus}=\{F^{\oplus}\mid{F\in\cF}\}$. 
Similarly for a model~$M^{\oplus}$ and a class of models~$\cM^{\oplus}$.

\begin{lemma}
$(W,R,S)\valid\{(A1),(A2),(A3)\}$ \ iff \ $R^+=S$.
\end{lemma}
\begin{proof}
This is a known fact.
Lemma~\ref{Lem:Segerberg:Axioms:Semantics} (in Appendix) gives more details.
\end{proof}

\begin{lemma}\label{Lem:Models:LP}
{\bf (a)} $\mL^\oplus \subseteq \Models(L^\BP)$.
\qquad
{\bf (b)} $\fL^\oplus = \Frames(L^\BP)$.
\end{lemma}
\begin{proof}
Any frame of the form $(W,R,R^+)$ validates $(A1),(A2),(A3)$.
\end{proof}

\begin{lemma}[Conservativity]\label{Lem:Conserv:BP}
For any consistent normal logic $L$, the logic $L^\BP$ is a conservative extension of~$L$:
if $A\in\Fm(\B)$ and ${L^\BP\proves A}$, then ${L\proves A}$.
\end{lemma}
\begin{proof}
If ${L\nproves A}$, then ${M_L\nmodels A}$ and ${M^\oplus_L\nmodels A}$. But ${M_L^\oplus\models L^\BP}$.
So ${L^\BP\nproves A}$.
\end{proof}

\subsection{Completeness for logics with the transitive closure modality}

In the proof of the main result, we will need to modify a valuation \emph{definably}.
By $A^\sigma$ we denote 
the application of a substitution $\sigma\colon\Var\to\Fm$ to a formula~$A$.

\begin{definition}
By a (\emph{modally}) \emph{definable variant} of a model $M=(F,V)$ we mean
a model of the form $M^\sigma=(F,V^\sigma)$, for some substitution~$\sigma$,
where the valuation $V^\sigma$ is defined by $V^\sigma(p)=V(p^\sigma)$, for every variable~$p$.
\end{definition}

In other words, ${M^\sigma\!,x\models p}$ iff ${M,x\models p^\sigma}$.
By induction one can easily prove:

\begin{lemma}
\label{Lem:Def:variant}
${M^\sigma\!,x\models A}$ \ iff \ ${M,x\models A^\sigma}$,
\ for all formulas~$A$.
\end{lemma}


Since a logic is closed under substitutions, we obtain the following fact.

\begin{lemma}\label{Lem:Msigma:L}
If $L$ is a logic and ${M\models L}$,
then ${M^\sigma\models L}$, for any substitution~$\sigma$.
\end{lemma}

Recall that the axioms (A1) and (A2) are canonical. In particular,
they are valid in the canonical frame of the logic $L^\BP$.
For (A3), this is not the case. However,
in order to obtain our completeness result, we do not necessarily need its
validity in the canonical frame. Instead,
we only need that
after taking a definable filtration of a model ${M\models L^\BP}$
(in particular, of~$M_L$, or any other model in which all substitution instances
of (A3) hold)
into a finite model $\fM$, the resulting frame $\fF$ 
satisfies the inclusion ${\fS\se(\fR)^+}$.
The following key lemma states that this
is indeed the case for the \emph{minimal} filtration.

Let us write ${M\models A^*}$ if we have ${M\models A^\sigma}$ for all substitutions~$\sigma$.

\begin{lemma}[Induction axiom and minimal filtration]\label{Lem:A3:min:filtration} \ \\
%
Let $M=(W,R,S,V)\models(A3)^*$ and let $\Phi\se\Fm$ be finite.
Then $S^{\mathsf{min}}_{\sim_\Phi}\se(R^{\mathsf{min}}_{\sim_\Phi})^+\!$.
\end{lemma}
\begin{proof}
Denote $r:=R^{\mathsf{min}}_{\sim_\Phi}$ and
$s:=S^{\mathsf{min}}_{\sim_\Phi}$.
To prove ${s\se r^+}$, assume $\fx\,s\,\fy$.
By definition of $R^{\mathsf{min}}_{\sim_\Phi}$, without loss of generality, $x\,S\,y$.
Consider ${Y:=r^+(\fx)\se\smash{\fW}}$. We need to show that ${\fy\in Y}$.

Since $\Phi$ is finite, every $\sim_\Phi$-equivalence class ${\fz\se W}$
is a definable (by some formula) subset of~$W$.
Since $Y$ is a finite collection of such subsets,
their union ${{\textstyle\bigcup}Y\se W}$ is also a definable subset of~$W$.
So, there is a formula $\Fi$
such that, for all ${z\in W}$, we have:
 ${M,z\models\Fi}$ iff ${z\in{\textstyle\bigcup}Y}$ iff ${\fz\in Y}$.

Firstly, $M\models\Fi\to\B\Fi$.
Indeed, if ${M,a\models\Fi}$, $a\,R\,b$, then ${\fa\in Y}$, $\fa\,r\,\fb$.
But $Y$ is closed under~$r$, hence
${\fb\in Y}$ and ${M,b\models\Fi}$.
Therefore, $M\models\BP(\Fi\to\B\Fi)$.

Secondly, $M,x\models\B\Fi$. Indeed, if $x\,R\,z$ then $\fx\,r\,\fz$, so ${\fz\in Y}$ and ${M,z\models\Fi}$.

Now we use that $M\models\BP(\Fi\to\B\Fi)\to(\B\Fi\to\BP\Fi)$.
Thus, ${M,x\models\BP\Fi}$. Now recall that $x\,S\,y$.
Then ${M,y\models\Fi}$, hence ${\fy\in Y}$.
\end{proof}

In Appendix (Lemma~\ref{Lem:A3:min:frame}) we strengthen the above lemma.

Now we come to the main technical tool of our paper.

\begin{theorem}[Transfer of ADF to logics with transitive closure]\label{Th:ADL:LBP}\ \\
If the class\/ $\mL$ admits definable filtration, then so does the class\/~$\Models(L^\BP)$.
\end{theorem}
\begin{proof}
{\sc Idea:}%
\footnote{The proof of the main theorem differs from the
proof of the corresponding Theorem~2.6 from our paper~\cite{KSZ:AiML:2014} in the following two aspects.
First, in~\cite{KSZ:AiML:2014} we filtrate a model of the form $M=(W,R,R^+,V)$
such that $(W,R,R^+)\valid L^\BP$, \ie, ${(W,R)\valid L}$;
 while here we will filtrate a model of the form $M=(W,R,S,V)$ such that ${M\models L^\BP}$.
As a consequence, in the old proof, we had to show that
${(R^+)^{\mathsf{min}}_\sim\se(R^{\mathsf{min}}_\sim})^+$,
which is quite simple, while here we need to show that
${\fS\se(\fR)^+}$, for this we need Lemma~\ref{Lem:A3:min:filtration}.
Secondly, we transform a filtration of $(W,R,V)$ through a set of formulas ${\Phi\se\Fm(\B)}$ into a
filtration of $(W,R,S,V)$ through some set of formulas $\Phi'\se\Fm(\B,\BP)$,
so we need to build $\Phi'$ from~$\Phi$.}
in order to filtrate a model $M=(W,R,S,V)\models L^\BP$
for $\Gamma\se\Fm(\B,\BP)$, we build a special set of formulas $\Delta\se\Fm(\B)$
and $\Delta$-filtrate the reduct $N=(W,R,V)\models L$ of~$M$ into a finite model
$\ff{N}=(\ff{W},\ff{R},\ff{V})\models L$. Then we 
show that 
$\ff{N}^+=(\ff{W},\ff{R},(\ff{R})^+,\ff{V})\models L^\BP$
is a $\Gamma$-filtration of~$M$.
A~subtlety is that we first take a modified valuation $V^\sigma$ and actually filtrate $N^\sigma$, not~$N$.

\smallskip
{\sc Formally:}
take a model $M=(W,R,S,V)$ such that ${M\models L^\BP}$
and a finite $\Sub$-closed set of formulas $\Gamma\se\Fm(\B,\BP)$.
For each formula ${\Fi\in\Gamma}$, fix a fresh (not occurring in~$\Gamma$) variable~$q_\Fi$.
Consider a substitution $\sigma\colon\Var\to\Fm(\B,\BP)$ defined by
${\sigma(q_\Fi)=\Fi}$ for all ${\Fi\in\Gamma}$ and ${\sigma(p)=p}$
for all other variables~$p$. In the definable variant $M^\sigma=(W,R,S,V^\sigma)$ of~$M$
we have: $M^\sigma\models{q_\Fi\lra\Fi}$ for all ${\Fi\in\Gamma}$ (since ${\Fi^\sigma=\Fi}$),
hence $M^\sigma\models{\B q_\Fi\lra\B\Fi}$ and even
$M^\sigma\models{A\lra A^\sigma}$, for any formula ${A\in\Fm(\B)}$.
We also have ${M^\sigma\models L^\BP}$ by Lemma~\ref{Lem:Msigma:L}.

Now consider the reduct $N^\sigma=(W,R,V^\sigma)$ of~$M^\sigma$.
Clearly, ${N^\sigma\models L}$.
Consider the following finite $\Sub$-closed set of $\B$-formulas:
\begin{center}$
\Delta \ := \ \{\,q_\Fi,\,\B q_\Fi\ \mid\ \Fi\in\Gamma\,\}\ \ \subset \ \ \Fm(\B).
$\end{center}
$\mL$ admits definable filtration, so
there is a $\Delta$-filtration $\ff{N^\sigma}=(\ff{W},\ff{R},\ff{V^\sigma})$
of $N^\sigma$
through some finite set $\Phi\se\Fm(\B)$ with ${\Delta\se\Phi}$
such that ${\ff{N^\sigma}\models L}$.
Let us change $\ff{V^\sigma}$ on the variables $p\in\Var(\Gamma)$
by putting:%
\footnote{In the model $\ff{N^\sigma}$, we will talk about the truth of formulas from~$\Delta$ only,
so the valuation $\ff{V^\sigma}$ of variables from $\Var(\Gamma)$ is irrelevant for this.
}
 ${\fx\models p}$ $\defeq$ ${\fx\models q_p}$.

{\sf Remark.} Since we will have several models on the same set of points,
we need a more subtle notation. In particular, we have $\ff{W}=W/{\sim}^{N^\sigma}_\Phi$,
this notation shows explicitly in which models we consider the $\sim_\Phi$-equivalence of points.

It remains to prove the following statement.

\medskip\noindent	
{\sf Claim.} {\it The model $\fM:=(\fW,\fR,(\fR)^+,\ff{V^\sigma})$
is a $\Gamma$-filtration (through $\Phi^\sigma\!$) of~$M$.}

\begin{description}
\item[{\sf(1)}] We show that $\fW=W/{\sim}^{M}_{\Phi^\sigma}$.
For any ${x\in W}$ and $A\in\Fm(\B)$, we have:
\begin{center}
$N^\sigma\!,x\models A$ \ \ $\Iff$ \ \ $M^\sigma\!,x\models A$ \ \ $\Iff$ \ \ $M,x\models A^\sigma$.
\end{center}

Therefore, for all $x,y\in W$, 
we have:
$(x\,\sim^{N^\sigma}_\Phi y)$ \ \ iff \ \ $(x\,\sim^{M}_{\Phi^\sigma}y)$.

This allows us to introduce a simpler notation $\sim$ for $\sim^{N^\sigma}_\Phi$ and $\sim^{M}_{\Phi^\sigma}$.

Since $\ff{N^\sigma}$ is a $\Delta$-filtration of~$N^\sigma\!$, we have:
$R^{\mathsf{min}}_\sim\ \se\ \fR\ \se\ R^{\mathsf{max}}_{\sim,\Delta}$. \hfill $(*)$

\item[{\sf(2)}]
The relation $\sim$ respects~$\Gamma$.
Indeed, $\Phi\su\Gamma'\su\{q_\Fi\mid \Fi\in\Gamma\}$, 
hence ${\Phi^\sigma\su\Gamma}$. 

\item[\sf(3)] Let us show that $M,x\models p$ $\Iff$ $\fM\!,\fx\models p$, \ for all $x\in W$ and $p\in\Var(\Gamma)$.
\begin{center}$
\begin{array}{lclcr}
       M,x\models p &
\Iff & M^\sigma\!,x\models q_p
& \Iff & N^\sigma\!,x\models q_p\\
&&&& \Updownarrow\qquad\\
\smash{\fM},\fx\models p
&\Iff&
\smash{\fM},\fx\models q_p
&\Iff&
\smash{\ff{N^\sigma},\fx}\models q_p
\end{array}$
\end{center}

\item[{\sf(4)}] $R^{\mathsf{min}}_\sim\ \se\ \fR$. This holds by~$(*)$.

\item[{\sf(5)}] $S^{\mathsf{min}}_\sim\ \se\ \fS$, where ${\fS:=(\fR)^+}$.
Using (4) and Lemma~\ref{Lem:A3:min:filtration}, we obtain:
\begin{center}
$S^{\mathsf{min}}_\sim$ \ $\se$ \ $(R^{\mathsf{min}}_\sim)^+$ \ $\se$ \ $(\fR)^+$ \ $=$ \ $\fS$.
\end{center}

\item[{\sf(6)}] $\fR\se R^{\mathsf{max}}_{\sim,\Gamma}$. Due to~$(*)$, it suffices to prove that 
$R^{\mathsf{max}}_{\sim,\Delta}\se R^{\mathsf{max}}_{\sim,\Gamma}$.

Assume that $\fx\,(R^{\mathsf{max}}_{\sim,\Delta})\,\fy$. To show that
$\fx\,(R^{\mathsf{max}}_{\sim,\Gamma})\,\fy$, take any ${\B\Fi\in\Gamma}$. Then: 
\begin{center}
\begin{tabular}{lclcl}
$M,x\models\B\Fi$ &
$\smash{\stackrel{\smash{(a)}}{\IFF}}$ &
$M^\sigma\!,x\models\B q_\Fi$ &
$\smash{\stackrel{\smash{(b)}}{\IFF}}$ &
$N^\sigma\!,x\models\B q_\Fi$\\
& & & & $\qquad\Downarrow{\scriptstyle(c)}$
\\
$M,y\models \Fi$ &
$\smash{\stackrel{\smash{(a)}}{\IFF}}$ &
$M^\sigma\!,y\models q_\Fi$ &
$\smash{\stackrel{\smash{(b)}}{\IFF}}$ &
$N^\sigma\!,y\models q_\Fi$
\end{tabular}
\end{center}
We used:
$(a)$ Lemma~\ref{Lem:Def:variant};
$(b)$ $q_\Fi,\B q_\Fi\in\Fm(\B)$;
$(c)$ $\B q_\Fi\in\Delta$ and $\fx\,(R^{\mathsf{max}}_{\sim,\Delta})\,\fy$.

\item[{\sf(7)}] $\fS\se S^{\mathsf{max}}_{\sim,\Gamma}$.\,
Due to~$(*)$, it suffices to prove that $(R^{\mathsf{max}}_{\sim,\Delta})^+\se S^{\mathsf{max}}_{\sim,\Gamma}$.

Let us denote $r:=R^{\mathsf{max}}_{\sim,\Delta}$ and $s:=S^{\mathsf{max}}_{\sim,\Gamma}$.
In order to prove that ${r^+\se s}$, it suffices to prove two inclusions:
${r\se s}$ and ${r\circ s\se s}$.

\medskip\noindent{\sf(7a)} \
Proof of ${r\se s}$. We will use the axiom ${\BP p\to\B p}$.

Assume $\fx\,R^{\mathsf{max}}_{\sim,\Delta}\,\fy$.
To prove that $\fx\,S^{\mathsf{max}}_{\sim,\Gamma}\,\fy$,
take any ${\BP\Fi\in\Gamma}$. 
Then:
\begin{center}
\begin{tabular}{l@{\ \ }c@{\ \ }l@{\ \ }c@{\ \ }l@{\ \ }c@{\ \ }l}
$M,x\models\BP\Fi$ &
$\smash{\stackrel{\smash{(d)}}{\TO}}$ & 
$M,x\models\B\Fi$ &
$\smash{\stackrel{\smash{(a)}}{\IFF}}$ &
$M^\sigma\!,x\models\B q_\Fi$ &
$\smash{\stackrel{\smash{(b)}}{\IFF}}$ &
$N^\sigma\!,x\models\B q_\Fi$\\
& & & & & & $\qquad\Downarrow{\scriptstyle(c)}$
\\
& & $M,y\models \Fi$ &
$\smash{\stackrel{\smash{(a)}}{\IFF}}$ &
$M^\sigma\!,y\models q_\Fi$ &
$\smash{\stackrel{\smash{(b)}}{\IFF}}$ &
$N^\sigma\!,y\models q_\Fi$
\end{tabular}
\end{center}
$(d)$ holds since $M\models{\BP\Fi\to\B\Fi}$.
The explanations of $(a,b,c)$ are the same.

\medskip\noindent{\sf(7b)} \
Proof of ${r\circ s\se s}$. We will use the axiom ${\BP p\to\B\BP p}$.

Assume $\fx\,R^{\mathsf{max}}_{\sim,\Delta}\,\fy\,S^{\mathsf{max}}_{\sim,\Gamma}\,\fz$.
To prove that $\fx\,S^{\mathsf{max}}_{\sim,\Gamma}\,\fz$,
take any ${\BP\Fi\in\Gamma}$. 
Then:
\begin{center}
\begin{tabular}{l@{\ \ }c@{\ \ }l@{\ \ }c@{\ \ }l@{\ \ }c@{\ \ }l}
$M,x\models\BP\Fi$ &
$\smash{\stackrel{\smash{(e)}}{\TO}}$ & 
$M,x\models\B\BP\Fi$ &
$\smash{\stackrel{\smash{(a)}}{\IFF}}$ &
$M^\sigma\!,x\models\B q_{\BP\Fi}$ &
$\smash{\stackrel{\smash{(b)}}{\IFF}}$ &
$N^\sigma\!,x\models\B q_{\BP\Fi}$\\
& & & & & & $\qquad\Downarrow{\scriptstyle(c)}$
\\
$M,z\models\Fi$ &
$\smash{\stackrel{\smash{(g)}}{\OT}}$ &
$M,y\models \BP\Fi$ &
$\smash{\stackrel{\smash{(a)}}{\IFF}}$ &
$M^\sigma\!,y\models q_{\BP\Fi}$ &
$\smash{\stackrel{\smash{(b)}}{\IFF}}$ &
$N^\sigma\!,y\models q_{\BP\Fi}$
\end{tabular}
\end{center}
We used:
$(e)$~$M\models{\BP\Fi\to\B\BP\Fi}$;
$(a)$~Lemma~\ref{Lem:Def:variant};
$(b)$~${\B q_{\BP\Fi}\in\Fm(\B)}$;
$(c)$~$\B q_{\BP\Fi}\in\Delta$ and $\fx\,(R^{\mathsf{max}}_{\sim,\Delta})\,\fy$;
$(g)$~${\BP\Fi\in\Gamma}$ and $\fy\,S^{\mathsf{max}}_{\sim,\Gamma}\,\fz$.
\end{description}

This completes the proof of theorem.
\end{proof}

Note that in {\sf(7a)} and {\sf(7b)} we proved inclusions that involve maximal relations,
and these inclusions resemble the axioms (A1) and (A2).
This is not a coincidence. In Lemma~4.3 of our paper~\cite{KSZ:AiML:2014},
we already made this observation for any \emph{right-linear grammar} axiom
and both (A1) and (A2) are right-linear.

\smallskip
Let us summarize the main result on logics with transitive closure.
We give two versions. The first uses the (rather unusual)
property that the models of~$L$ admit filtration.
The second uses the filtration of frames of~$L$,
but has an additional requirement of canonicity.

\begin{theorem}[Main result, version 1]\label{Th:Main:result:1}
Assume that the class of models $\mL$ of a logic~$L$ admits definable filtration. Then:
\begin{enumerate}\compact
\item[\rm(1)] the class of models $\Models(L^\BP)$ admits definable filtration;
\item[\rm(2)] hence the logic $L^\BP$ has the finite model property;
\item[\rm(3)] hence the logic $L^\BP$ is Kripke complete.
\end{enumerate}
\end{theorem}

The above theorem allows us to `iterate' the ADF property, see Section~\ref{Sect:mainCorollary}.

\begin{theorem}[Main result, version 2]\label{Th:Main:result:2}
Assume that a logic $L$ is canonical and the class of its frames $\fL$
 admits definable filtration. Then:
\begin{enumerate}\compact
\item[\rm(1)] the class $\Models(L^\BP)$ 
admits definable filtration;
\item[\rm(2)] hence the logic $L^\BP$ has the finite model property;
\item[\rm(3)] hence the logic $L^\BP$ is Kripke complete.
\end{enumerate}
\end{theorem}

%% file: 4_Locally_finite.tex
\section{
$\PDL$ization of logics that admit filtration
}\label{Sect:pdlization}

\newcommand{\regSymb}{\sharp}
\newcommand{\extReg}[1]{{#1}{\lefteqn{}}^\regSymb}
\newcommand{\extRegn}[2]{{#1}{}^{(#2)}}

\subsection{Main corollary}\label{Sect:mainCorollary}
Now we apply Theorem  \ref{Th:Main:result:1} to show 
that 
if $\mL$ admits definable filtrations, then 
the following $\PDL$-like expansions of $L$ have the
finite model property.

\begin{definition}
For an alphabet $\Sigma$, let $\extReg{\Sigma}=\Sigma\cup \{(e\circ f), (e\cup f), e^+
\mid e, f\in \Sigma\}$,
assuming that the added symbols are not in $\Sigma$.
Put $\extRegn{\Sigma}{0}=\Sigma$,  $\extRegn{\Sigma}{n+1}=\extReg{(\extRegn{\Sigma}{n})}$.

For a frame $F=(W,(R_e)_{e\in\Sigma})$, put 
$\extReg{F}=(W,(R_e)_{e\in\extReg{\Sigma}})$, where 
for $e,c\in \Sigma$, 
\begin{center}
$
R_{e\circ c}=R_e\circ R_c, \quad
R_{e\cup c}=R_e\cup R_c,\quad
R_{e^+}=(R_e)^+.
$\end{center}
Put $\extRegn{F}{0}=F$,  $\extRegn{F}{n+1}=\extReg{(\extRegn{F}{n})}$.

For a model $M=(F,V)$, we put $\extReg{M}=(\extReg{F},V)$ and $\extRegn{M}{n}=(\extRegn{F}{n},V)$.

For a logic $L$ over $\Sigma$, let $\extReg{L}$ be the smallest (normal) logic 
over $\extReg{\Sigma}$ that contains $L$ and 
the following \PDL-like axioms, for all $e,c\in \Sigma$:
\begin{center}$\begin{array}{l}
 \BX{e\cup c}p\lra\BX{e}p\wedge \BX{c}p, \\
 \BX{e\circ c}p\lra\BX{e}\BX{c}p,  \\
 \BX{e^+} p \to \BX{e} p, \quad  
 \BX{e^+}p \to \BX{e}\BX{e^+} p, \quad 
 \BX{e^+}(p\to \BX{e}p)\to (\BX{e}p \to \BX{e^+} p).
\end{array}$\end{center}

We put $\extRegn{L}{0}=L$,  $\extRegn{L}{n+1}=\extReg{(\extRegn{L}{n})}$. 
\end{definition}


The following is a simple analogue of Lemma~\ref{Lem:Models:LP}.

\begin{lemma}\label{Lem:Models:Lstar}
{\bf (a)} ${M\mo L}$ implies ${\extReg{M}\mo \extReg{L}}$.
\quad
{\bf (b)} ${F\mo L}$ \ iff \ ${\extReg{F}\mo \extReg{L}}$.
\end{lemma}

By an easy induction on $n$, we obtain
\begin{proposition}
For a frame $F$ and $n<\omega$,  
$
F\mo L \quad \text{iff} \quad \extRegn{F}{n}\mo \extRegn{L}{n}
$.
\end{proposition} 

\begin{proposition}
For a logic $L$ and $n<\omega$, $\extRegn{L}{n}$ is conservative over $L$.
\end{proposition}
\begin{proof}
As in Lemma~\ref{Lem:Conserv:BP}, using $\extRegn{M_L}{n}$ instead of $M_L^\oplus$
and Lemma~\ref{Lem:Models:Lstar}(a).
\end{proof}

\begin{lemma}\label{ref:UnionComp}
Let $L$ be a logic over~$\Sigma$, $e,c\in \Sigma$. Let $L_1$ and $L_2$ be the logics 
over ${\Sigma\cup\{g\}}$, where ${g\notin\Sigma}$,
such that
\begin{center}
\begin{tabular}{l}
$L_1$ extends $L$ with the axiom \ $\BX{g}p\lra\BX{e}p\wedge \BX{c}p$,\\
$L_2$ extends $L$ with the axiom \ $\BX{g}p\lra\BX{e}\BX{c}p$.
\end{tabular}
\end{center}
If $\mL$ admits definable filtration, then so do $\Models(L_1)$ and $\Models(L_2)$.
\end{lemma}
\begin{proof}
Straightforward.
Details can be reconstructed from the proof of Lemma 2.3 in \cite{KSZ:AiML:2014},
which is the analog of our lemma for the classes of frames.
\end{proof}

\begin{theorem}\label{thm:iterated}
Let $L$ be a logic over a finite alphabet~$\Sigma$.
If the class of its models $\mL$ admits definable filtration, then, for every $n<\omega$, we have:
\begin{enumerate}
\item $\Models(\extRegn{L}{n})$ admits definable filtration.
\item $\extRegn{L}{n}$ has the finite model property; a fortiori, $\extRegn{L}{n}$ is Kripke complete.
\item If $L$ is finitely axiomatizable, then $\extRegn{L}{n}$ 
is decidable. 
\item If the class of finite frames of $L$ is decidable, then 
$\extRegn{L}{n}$  is co-recursively enumerable. 
\end{enumerate} 
\end{theorem}
\begin{proof} 
(i) By Theorem \ref{Th:ADL:LBP} and Lemma \ref{ref:UnionComp},
if $\mL$ admits definable filtration, then so does
$\Models(\extReg{L})$.
So, (i) follows by induction on~$n$.

(ii) By Lemma~\ref{Lem:AF:Mo:FMP}.

(iii) Note that if 
$L$ is finitely axiomatizable, then so is $\extRegn{L}{n}$.
The claim then follows from Harrop's Theorem (see Section~\ref{Sect:Prelim}).

(iv) If the class of finite frames of $L$ is decidable, then the class 
of finite frames of $\extRegn{L}{n}$ is decidable, too.
In this case $\extRegn{L}{n}$  is co-recursively enumerable, since $\extRegn{L}{n}$ is the logic of its finite frames.
\end{proof}


\subsection{Expansions of locally finite logics}\label{Sect:Loc:Finite}

 
For $k\leq \omega$, a {\em $k$-formula} is a formula in proposition letters $p_i$, $i<k$.
Recall that a logic is called \emph{locally finite} (or \emph{locally tabular}),
if, for every $k<\omega$, there exist
only finitely many $k$-formulas up to the equivalence in~$L$.


\begin{theorem}
If $L$ is locally tabular, then $\mL$ admits definable filtration.
\end{theorem}
\begin{proof}
Let $M$ be a model, $M\mo L$, $\Gamma$ a finite $\Sub$-closed set of formulas. 
For some $k<\omega$, every formula in $\Gamma$ is a $k$-formula.
Let $\Phi$ be the set of all $k$-formulas. 
Consider the maximal filtration $\ff{M}$ of $M$ through $\Phi$ (in \cite{Shehtman:AiML:2014},
such filtrations are called {\em canonical}).
Since $\vL$ is locally tabular, $\ff{M}$ is finite. In this case
$\ff{M}$ is a p-morphic image of $M$ (for details, see \cite[Proposition 2.32]{Shehtman:AiML:2014}).
Hence $\ff{M}\mo L$, as required.
\end{proof}

\begin{corollary}
If $L$ is locally tabular, then $\extRegn{L}{n}$ has the FMP, for every ${n<\omega}$.
\end{corollary}


\subsection{Agents that admit filtration}

Here we consider a special kind of definable filtration, called \emph{strict filtration}.

\begin{definition}
If, in terms of Definition \ref{Def:Filtration}, 
$\sim\,=\,\sim_\Gamma$, then we call the filtration  
$\ff{M}$ \emph{strict}. 
The corresponding notions ``a class of frames (or models) \emph{admits strict filtration}''
are introduced in the obvious way.
\end{definition}

Strict filtration is the most standard variant of filtration; it is well-known that the classes of frames of  the logics $\bK,\mathbf{T},\mathbf{K4},\mathbf{S4},\mathbf{S5}$ admit strict filtration
(for the logics $\bK$ and $\mathbf{T}$, even the \emph{minimal} strict filtration works;
for $\mathbf{K4}$, $\mathbf{S4}$, $\mathbf{S5}$, strict filtration is obtained by
taking the
transitive closure of the minimal filtered relation \cite{Segerberg1968}). 
\ISH{DC this ref}

\smallskip
Let us recall the notion of the {\em fusion} of logics.
Let $L_1,\ldots,L_k$ be logics over  finite alphabets $\Sigma_1, \ldots, \Sigma_k$.
Without loss of generality we assume that these alphabets are disjoint.
The {\em fusion}
$L_1*\ldots * L_k$ of these logics is the smallest
normal logic over the alphabet $\Sigma=\Sigma_1\cup \ldots \cup \Sigma_k$ that
contains $L_1\cup \ldots \cup L_k$.

It is well-known that the fusion operation preserves
Kripke completeness, the finite model property, and decidability  \cite{WolterKracht1991}.
We observe that it also preserves the admits strict filtration property.
\def\cG{\mathcal{G}} 
\begin{theorem}[Fusion and strict filtration]\label{Th:Fusion:Strict:Filtration}
If classes of frames $\Frames(L_i)$, $1\leq i\leq k$,  admit strict filtration,
then $\Frames{(L_1*\ldots *L_k)}$ admits strict filtration. 
\end{theorem} 
\begin{proof}
The idea is the same as in the proof of Theorem \ref{Th:ADL:LBP}.
To simplify notation, we consider the case of unimodal logics.
Let $L={L_1*\ldots * L_k}$, 
$M=(F,V)$ be a model on an $L$-frame
$F=(W,R_1,\ldots, R_k)$, $\Gamma\se\Fm(\B_1,\ldots,\B_k)$ be finite 
and $\Sub$-closed.
For $\vf\in \Gamma$, we take fresh variables $q_\vf$, 
and consider a model $M'=(F,V')$ such that 
\begin{center}
$M,x\mo\vf \ \text{ iff } \ M',x\mo \vf \ \text{ iff } \ M',x\mo q_\vf$
\end{center}
 for all $x$ in $M$. 
For $1\leq i\leq k$, we put:
\begin{center}$
\Gamma_i=\{q_\vf\mid \vf\in \Gamma\} \cup
\{\B q_\vf \mid \B_i \vf\in \Gamma\};
$\end{center}
remark that $\Gamma_i\se\Fm(\B)$.  
Let $\sim_i$ be the equivalence induced by $\Gamma_i$ in the model 
$M_i=(W,R_i,V')$, and   $\sim_\Gamma$
the equivalence induced by $\Gamma$ in $M$.  
Observe that 
\begin{center}$
M_i,x\mo \B q_\vf \ \text{ iff } \ M,x\mo \B_i \vf \ \text{ for all } \ \vf\in\Gamma. 
\lefteqn{\qquad\qquad\qquad(*)}
$\end{center}
Therefore, 
one can see
that $\sim_i\;=\;\sim_\Gamma$ for all $i$. 
Put $\fW=W{/}{{\sim}_\Gamma}$.
For each $i$, there exists a filtration $\fM_i=(\fW,\ff{R_i},\fV_i)$ of $M_i$ through $\Gamma_i$ such that 
$(\fW,\ff{R_i})\mo L_i$.
The valuations $\fV_i$ coincide on the variables $q_\vf$.
W.l.o.g., they also coincide on other variables (since they do no occur in $\Gamma_i$),
and that $\fM_i,\ff{x}\mo p$ iff $M,x\mo p$ for each variable ${p\in \Gamma}$.
The resulting valuation on $\fW$ is denoted by~$\fV$.  

Consider the model  $\fM=(\fW,\ff{R_1},\ldots,\ff{R_k},\fV)$. 
Note that its frame validates the fusion $L$. 
We claim that $\ff{M}$ is a filtration of $M$ through $\Gamma$. 
Clearly, $\ff{R}_i$ contains the $i$-th minimal filtered relation. 
To check that $\ff{R_i}$ is contained in the $i$-th maximal filtered relation, assume that 
$\ff{x}\ff{R_i}\ff{y}$, $M,x\mo\B_i\vf$, and 
$\B_i\vf\in \Gamma$.
Then $M_i,x\mo\B q_\vf$, by~(*).
Since $\ff{M_i}$ is a filtration of $M_i$ through $\Gamma_i$ and ${\B q_\vf\in \Gamma_i}$,
we have $\fM_i,\ff{y}\mo q_\vf$. By Filtration lemma, ${M_i,y\mo q_\vf}$. Hence, ${M',y\mo q_\vf}$
and we conclude that ${M,y\mo\vf}$, as required.
\end{proof}

\begin{theorem}\label{thm:agents}
Let $L_1,\ldots,L_k$ be canonical logics 
and their classes of frames
$\Frames(L_i)$, $1\leq{i\leq k}$, admit strict filtration.   
Then, for every $n<\omega$, the logic
$\extRegn{(L_1*\ldots * L_k)}{n}$ has the finite model property. 
\end{theorem}
\begin{proof}
The fusion $L=L_1*\ldots * L_k$ is canonical. 
By Theorem~\ref{Th:Fusion:Strict:Filtration}, the class $\fL$ admits strict filtration.
Hence $\mL$ admits definable filtration, by Theorem \ref{Th:ADF:Fr:Mod}. 
Finally, $\extRegn{(L)}{n}$ has the FMP, by Theorem \ref{thm:iterated}.
\end{proof}

%% file: 5_MFP_formulas.tex

\subsection{A class of formulas that admit strict filtration}

We present a collection of modal formulas that admit strict (and so definable) filtration.
The obvious candidates are modal formulas whose first-order equivalents belong
to a certain FO fragment we call $\MFP$.%
\footnote{The abbreviation stems from ``\emph{p}reserved under \emph{m}inimal \emph{f}iltration''.}
 We define it inductively
as the minimal set of FO formulas satisfying the following conditions:
%
%
\begin{itemize}
\item if $x$ and $y$ are variables, $R$ is a binary relation symbol, then  $R(x,y) \in \MFP$ and $x = y \in \MFP$;
\item if $A$ and $B$ are in $\MFP$, then $(A\land B)$ and $(A\lor B)$ are in $\MFP$;
\item if $A\in \MFP$, and $v$ is a variable, then $\forall v\, A$ and $\exists v\, A$ are in $\MFP$;
\item if $x$ and $y$ are variables, $R$ is a binary relation symbol, and ${A\in\MFP}$,
then $\forall x \forall y (R(x,y) \to A)$ and $\forall x \forall y (x = y \to A)$ are in $\MFP$.
\end{itemize}

This definition is the restriction of the fragment $\mathsf{POS} + \forall \mathsf{G}$ from \cite{Libkin14}
to the first-order language with only binary predicates. 
Examples of $\MFP$-sentences are reflexivity $\forall x R(x,x)$, symmetry $\forall x \forall y (R(x,y) \to R(y,x))$, and density
$\forall x \forall y (R(x,y) \to \exists z\, (R(x,z) \land R(z,y)))$, but not transitivity.

FO counterparts of minimal filtrations
are strong onto homomorphisms. 


\begin{definition}
Given two frames $F = (W, R)$ and $F'=(W'\!,R')$,
a map $h\colon {W \to W'}$
is a \emph{strong onto homomorphism} if the following conditions hold:
\begin{itemize}
\item $h$ is onto;
\item for all $x,y\in W$, \, if $x\,R\,y$, then $h(x)\,R'\,h(y)$ (\emph{monotonicity});
\item for all $x',y' \in W'$, if $x'\,R'\,y'$, then
then there exist $x,y \in W$ such that ${h(x) = x'}$, ${h(y) = y'}$, and $x\,R\,y$ (\emph{weak lifting}).
\end{itemize}
\end{definition}

Note that a strong homomorphism $h$ from $F$ onto $F'$ induces
an equivalence $\sim$ on~$W$ defined by
$x\sim y$ iff ${h(x)=h(y)}$, and then
$F'$ is isomorphic to the 
\emph{minimal filtrated frame} $F^{\mathsf{min}}_{\sim}=(W/{\sim},R^{\mathsf{min}}_{\sim})$.
Conversely, if a model $\fM$ is a minimal filtration of a model~$M$,
then the \emph{filtration map} ${h(x)=\fx}$
(this term is also used in~\cite[Definition~2.27]{Shehtman:AiML:2014})
is a strong homomorphism from $F$ onto~$\fF$.

In \cite[Proposition~5.2]{Libkin14}, it was shown that $\MFP$-formulas are preserved under strong onto homomorphisms.
Moreover, any FO formula that is preserved under strong onto homomorphisms is equivalent
to some $\MFP$-formula~\cite{itas2015}.

\begin{definition} A modal formula $\Fi$ is called a \emph{modal\/ $\MFP$-formula}
if it has a FO equivalent (on frames) from~$\MFP$.
\end{definition}

Expressions of the the form $p \land \Diamond q \to \psi$,
where $\psi$ is a positive modal formula,
are typical examples of modal $\MFP$-formulas.
Note that these examples are Sahlqvist formulas, and hence canonical.

\begin{theorem}\label{MFP-FO}
For any set $\Phi$ of modal $\MFP$-formulas over a finite alphabet~$\Sigma$,
the class of frames $\Frames(\bK_\Sigma+\Phi)$ admits strict filtration.
\end{theorem}
\begin{proof}
Denote $\cF = \Frames(\bK_\Sigma+\Phi)$.
Let $M=(F,V)$ be an $\cF$-model and $\Gamma$ a finite \Sub-closed set of formulas. 
Take the minimal filtration $\fM=(\fF,\fV)$ of $M$ through~$\Gamma$;
note that this filtration is strict.
Then the filtration map $h(x)=\fx$ is a strong homomorphism from $F$ onto~$\fF$.
Since the set $\Phi^*$ of MFP first-order equivalents of $\Phi$ is true in~$F$, 
it is also true in~$\fF$. Hence $\fM$ is an $\cF$-model.
\end{proof}

%

\hide{
\begin{theorem}
For any set $\Phi$ of canonical modal $\MFP$-formulas in a finite alphabet $\Sigma$ $\mathsf{Mod}(\mathbf{K}_\Sigma+\Phi)$ admits strict
\ISH{????}
filtrations.
\end{theorem}  
\begin{proof}
Since $\Phi$ is a set of canonical formulas, due to Theorem~\ref{Th:ADF:Fr:Mod},
it is sufficient to check that the class of frames for $\mathbf{K}_\Sigma+\Phi$ admits definable filtration.
This follows from Theorem~\ref{MFP-FO}.
\end{proof}
}

From Theorem \ref{thm:agents}, we obtain: 
\begin{corollary}
Let each $L_1,\ldots,L_k$ be any of the logics
$\bK,\mathbf{T},\mathbf{K4},\mathbf{S4},\mathbf{S5}$,
or a logic axiomatized by canonical modal $\MFP$-formulas.
Then, for any $n<\omega$, the logic
$\extRegn{(L_1*\ldots * L_k)}{n}$ has the finite model property.
\end{corollary}
\ISH{More examples? What about $K5$ and $K45D$? Do they admit strict filtration? }

%% file: 7_Open_conslusion.tex
\section{Conclusions and further reseach}

We proved that if $L$ is a canonical logic, and the class of its models $\fL$
admits definable filtration, then the logic $L^\BP$ is Kripke complete and, moreover,
has the FMP (and is decidable, if $L$ was finitely axiomatizable).
The first problem we pose is whether we can weaken the pre-conditions and 
obtain completeness of $L^\BP$ without obtaining the FMP.

\medskip\noindent
{\bf Problem 1.} {\it If $L$ is canonical, then is $L^\BP$ Kripke complete?}

\medskip
The second problem deals with the possible weakening the `canonicity'
condition to just `completeness' in our result.

\medskip\noindent
{\bf Problem 2.} {\it If the logic $L$ is complete and the class of its frames $\fL$ admits definable filtration,
then does the same holds for the logic $L^\BP$?}

\medskip
The following questions have a more technical character.

\medskip\noindent
{\bf Question 1:} if we replace in the logic $\bK^\BP$
the axiom $(A2)$ ${\BP p\to\B\BP p}$ with $(A2')$ ${\BP p\to\BP\B p}$,
then will the resulting logic be complete?
Clearly, the frames of this logic are the same as for $\bK^\BP$,
so if it is complete, it must coincide with~$\bK^\BP$.
More concrete, does the logic with $(A2')$ derive $(A2)$?

\medskip\noindent
{\bf Question 2:} {\it Is the logic $\mathbf{K.2}^\BP$ Kripke complete?} (We conjecture: yes.)

Recall that the logic $\mathbf{K.2}$ extends $\bK$ with the formula $\D\B p\to\B\D p$.
It is canonical and hence
complete with respect to the class of frames $(W,R)$
that satisfy the following first-order \emph{convergence} (or Church--Rosser) condition:
\begin{center}
$\forall x,y,z\,(x\,R\,y\ \land x\,R\,z\ \To\ \exists w\,(y\,R\,w\ \land z\,R\,w))$.
\end{center}
Our main result is not applicable to this logic,
since the class of its frames $\Frames(\mathbf{K.2})$
does not admit filtration, as we shown in \cite[Theorem~5.4]{KSZ:AiML:2014}.

One can easily see that if the relation $R$ is convergent,
then so is its transitive closure~$R^+$.
So it is natural to attempt to derive the
formula $\DP\BP p\to\BP\DP p$ in $\mathbf{K.2}^\BP$.
We succeeded in deriving it (see Lemma~\ref{Lem:Convergence:TrCl} in Appendix).

\medskip\noindent
{\bf Question~3.}
In Lemma~\ref{Lem:A3:min:frame},
the bimodal formula $\BP(p\to\B p)\to (\B p \to \BP p)$
is shown to have the following property crucial for our main result:
{\it if all its substitution instances are true in some model $M=(F,V)$,
then this formula is valid on the frame of every definable minimal filtration:
if ${M\models A^*}$ then ${F^{\mathsf{min}}_{\sim\Phi}\valid A}$,
for any finite set of formulas~$\Phi$.}
Are there any other examples of such formulas?
How is this property related to the admissibility of filtration,
completeness, decidability of a logic axiomatized by such formulas?

%% file: 8_Appendix.tex
\Appendix

\refstepcounter{section}



\subsection{On filtration of the canonical model of a theory of a model}

\begin{lemma}[Filtration and canonical mapping]
\label{Lem:Filtr:canon:map:App}
Let $M=(W,R,V)$ be a model, $M_T=(W_T,R_T,V_T)$ the canonical model of its theory $T=\Th(M)$,
and $t\colon M\to M_T$ the canonical mapping: $t(a)=\Th(M,a)\in W_T$,
for ${a\in W}$.

Then, for any finite set of formulas $\Phi$, we have a bijection
between the (finite) quotient sets
$W/{\sim_\Phi}$ and $W_T/{\sim_\Phi}$ defined, for ${a\in W}$, by
\begin{center}
$f([a]_{\sim_\Phi})\ :=\ [t(a)]_{\sim_\Phi}$.
\end{center}
\end{lemma}
\begin{proof}
We denote $\ff{a}:=[a]_{\sim_\Phi}$.
Note that $\ff{x}=\ff{y}$ iff ${x\cap\Phi}={y\cap\Phi}$, for all ${x,y\in W_T}$.
Hence, by definition of~$f$,
for all ${a\in W}$ and ${x\in W_T}$, we have
\begin{center}
$f(\ff{a})=\ff{x}$ \ \ $\IFF$ \ \ $t(a)\cap\Phi=x\cap\Phi$.
\end{center}

First, let us show that $f$ is well-defined and injective: for all ${a,b\in W}$:
\begin{center}
$\ff{a}=\ff{b}$
 \ $\Iff$ \ $a\sim_\Phi b$
 \ $\Iff$ \ $\Th(M,a)\cap\Phi=\Th(M,b)\cap\Phi$
 \ $\Iff$ \ $[t(a)]_{\sim_\Phi}=[t(b)]_{\sim_\Phi}$.
\end{center}

To prove
that $f$ is surjective,
take any $\ff{x}\in(W_T/{\sim_\Phi})$.
Denote $A:=\AND(x\cap\Phi')$, where $\Phi'=\Phi\cup\{\neg B\mid B\in\Phi\}$.
Clearly, ${A\in x}$. Now ${M\nmodels\neg A}$, for otherwise $\neg A\in\Th(M)={T\se x}$
and $x$ is inconsistent.

Thus, $A$ is satisfiable in~$M$, so ${M,a\models A}$ for some ${a\in W}$.
We claim that ${f(\ff{a})=\ff{x}}$, \ie, for all ${B\in\Phi}$, we have ${M,a\models B}$ iff ${B\in x}$.
If ${B\in x}$, then $B\in(x\cap\Phi')$, so ${M,a\models B}$.
If ${B\notin x}$, then $\neg B\in(x\cap\Phi')$, so ${M,a\models\neg B}$.
\end{proof}

\subsection{On differentiated filtration}

\begin{lemma}\label{Lem:Differentiated}
Assume that $\mL$ admits (definable) filtration. Then
for every finite $\Sub$-closed set of formulas $\Gamma$
and every model $M\in\mL$,
there exists a (definable)
$\Gamma$-filtration $\fM\in\cM$ of~$M$ that is a differentiated model.
\end{lemma}
\begin{proof}
{\sc Idea:} first, build a $\Gamma$-filtration $M_1$ of $M$,
then a $\Fm$-filtration $M_2$ of $M_1$; finally,
build a differentiated $\Gamma$-filtration $\fM$ of~$M$ that is isomorphic to~$M_2$.

{\sc Formally}, let $M=(W,R,V)$, ${M\models L}$, and let $\Gamma$ be as stated above.

(1) Since $\mL$ admits filtration, there is a $\Gamma$-filtration $M_1=(W_1,R_1,V_1)$ of $M$
with ${M_1\models L}$. So, $W_1=W/{\sim}$ for some equivalence relation $\approx$ of finite index,
$\sim$ respects~$\Gamma$,
$R^{\mathsf{min}}_{\approx}\se R_1\se R^{\mathsf{max}}_{\approx,\Gamma}$,
$V_1$~is defined canonically on~$\Var(\Gamma)$.

(2) Let $M_2=(W_2,R_2,V_2)$ be a filtration
of~$M_1$ through the set of all formulas.%
\footnote{In fact, if a filtration through the set of all formulas
is finite, then it is unique,
\ie, the minimal and the maximal relations coincide.
But here we do not need this fact.}
So, $W_2=W_1/{\equiv}$, where $\equiv$ is the \emph{modal equivalence} relation;
$V_2$~is canonical on all variables.
By the Filtration lemma~\ref{Lemma:Filtration},
${M_1\equiv M_2}$, so ${M_2\models L}$.

(3) Now we build a model $\fM=(\fW,\fR,\fV)$ isomorphic to~$M_2$ as follows.
Put $\fW:=W/{\sim}$, where, for all $x,y\in W$, we define an equivalence relation $\sim$ by
\begin{center}
$x\,\sim\,y$ \quad $\smash{\stackrel{\mathrm{def}}{\IFF}}$ \quad
$(M_1,[x]_\approx)\,\equiv\,(M_1,[y]_\approx)$ \quad $\IFF$ \quad
$[[x]_\approx]_\equiv\ = \ [[y]_\approx]_\equiv$.
\end{center}

\noindent
{\bf Claim 1.} {\it The function $h([x]_\sim)=[[x]_\approx]_\equiv$ is a bijection between $\fW$ and $W_2$.}

\smallskip
\noindent{\sf Proof.} Easy.
This does not rely on the fact that $\approx$ and $\equiv$ are of finite index.

From now on, we denote $\fx=[x]_\sim$.

\smallskip\noindent
{\bf Claim 2.} {\it The equivalence relation $\sim$ on~$W$ respects $\Gamma$:
if $x\sim y$, then $x\,\sim_\Gamma\,y$.}

\smallskip\noindent
{\sf Proof.}
If $x,y\in W$ and $x\sim y$ then, by the Filtration lemma~\ref{Lemma:Filtration}, we have:
\begin{center}
$(M,x) \ \ \sim_\Gamma \ \ (M_1,[x]_\approx) \ \ 
\sim_\Fm\ \ 
(M_1,[y]_\approx) \ \ \sim_\Gamma \ \ (M,y)$.
\end{center}

Using the bijection~$h$, we transfer $R_2$ and $V_2$ to $\fM$ in the obvious way:
\begin{center}
$\fx\,\fR\,\fy$ \ $\smash{\stackrel{\mathrm{def}}{\IFF}}$ \ 
$h(\fx)\,R_2\,h(\fy)$;
\qquad
$\fx\models q$ \ $\smash{\stackrel{\mathrm{def}}{\IFF}}$
\ $M_2,h(\fx)\models q$,\ \mbox{for all $q\in\Var$.}
\end{center}

Since the models $\fM$ and $M_2$ are isomorphic, we have ${\fM\models L}$.

\smallskip\noindent
{\bf Claim 3.} {\it $\fV$ is canonical on each $p\in\Var(\Gamma)${\rm:}
\ $M,x\models p$ \ $\Iff$ \ $\fM,\fx\models p$.}

\smallskip\noindent
{\sf Proof.} Indeed: $(M,x) \sim_\Gamma (M_1,[x]_\approx) \sim_\Fm (M_2,[[x]_\approx]_\equiv)
\sim_{\Var} (\fM,\fx)$.

\smallskip\noindent
{\bf Claim 4.} {\it The inclusions \
$R^{\mathsf{min}}_{\sim}\ \se\ \fR\ \se\ R^{\mathsf{max}}_{\sim,\Gamma}$ \ hold.}

\smallskip\noindent
{\sf Proof.} (min) Clearly, $x\,R\,y$ $\To$ $[x]_\approx\,R_1\,[y]_\approx$ $\To$
$[[x]_\approx]_\equiv\,R_2\,[[y]_\approx]_\equiv$ $\Iff$ $\fx\,\fR\,\fy$.
\\
(max) If $\fx\,\fR\,\fy$,
then $[[x]_\approx]_\equiv\,R_2\,[[y]_\approx]_\equiv$.
But $R_2\se(R_1)^{\mathsf{max}}_{\equiv,\Fm}$.
So, for ${\B A\in\Gamma}$,
	
$M,x\models\B A$ $\Iff$ $M_1,[x]_\approx\models\B A$ $\To$
$M_1,[y]_\approx\models A$ $\Iff$ $M,y\models A$.

\smallskip\noindent
{\bf Claim 5.} {\it If $M_1$ is a definable filtration of~$M$, then $\fM$ is definable too.}

\smallskip\noindent
{\sf Proof.}
We use the following\ 
{\sf Fact.} {\it Let $M$ be a model and $\sim$ an equivalence relation on~$W$ of finite index.
Then $\sim$ is of the form $\sim_\Phi$
for some finite set of formulas $\Phi$ iff each equivalence class $[x]_\sim$
is defined in~$M$ by some formula.}

\smallskip
Indeed, if $\Phi$ is finite, then every class $[x]_{\sim_\Phi}\se W$ is defined by the formula
\begin{center}
$\AND\bigl(
\{\Fi\mid \Fi\in\Phi\ \mbox{and}\ M,x\models\Fi\}
\cup
\{\neg\Fi\mid \Fi\in\Phi\ \mbox{and}\ M,x\models\neg\Fi\}
\bigr)$.
\end{center}
Conversely, if $\sim$ partitions $W$ into 
finitely many classes and each class is defined by some formula $A_i$, $1\le i\le n$,
then clearly $\sim=\sim_\Phi$ for $\Phi=\{A_1,\ldots,A_n\}$.

\medskip
To prove Claim~5, assume $M_1$ is a filtration of $M$ through a finite~$\Phi$.
Then each class is defined by some formula~$A_i$.
In $\fM$, each $\sim$-class is obtained as the union of some $\sim_\Phi$-classes
(namely, those that are modally equivalent as points in~$M_1$).
Hence, each $\sim$-class is defined by the disjunction
of some formulas~$A_i$.
\end{proof}


\subsection{On the logic of convergent frames}

For convenience, we repeat the axioms for the transitive closure modality:
\begin{center}$
\mathsf{(A1)}\ \BP p \to \B p,
\qquad
\mathsf{(A2)}\ \BP p \to \B\BP p,
\qquad
\mathsf{(A3)}\ \BP(p\to\B p)\to (\B p \to \BP p).
$\end{center}
Note that in any logic 
$L^\BP$,
the following inference rule is derivable:
$$\smash{
\frac{\Fi\to\B \Fi}{\Fi\to\BP \Fi}}
\eqno{\mathsf{(R\BP)}}
$$

\noindent Indeed, here is a derivation:\\
1) $\Fi\to\B \Fi$.
2) $\BP(\Fi\to\B \Fi)$. 
3) $\B \Fi\to\BP \Fi$ by~$\mathsf{(A3)}$.
4) $\Fi\to\BP \Fi$ from 1 and~3.

\smallskip
Furthermore, in any logic $L^\BP$, the following formula is derivable:
$$
\B p\ \land\ \BP\B p \ \to\ \BP p
\eqno{\mathsf{(A\BP)}}
$$
since one of its premises, $\BP\B p$, is stronger then the premise $\BP(p\to\B p)$ in~$\mathsf{(A3)}$.

\medskip
Recall that the logic $\mathbf{K.2}$ extends $\bK$ with the axiom $\D\B p\to\B\D p$.

\begin{lemma}[Convergence for transitive closure]\label{Lem:Convergence:TrCl}
$\!\!\mathbf{K.2}^\BP\!\proves \DP\BP p\to\BP\DP p$.
\end{lemma}
\begin{proof}
The proof is in two stages.

\smallskip\noindent{\bf (1)} We derive $\D\BP p\too\BP\D p$, using $\D\B\Fi\to\B\D\Fi$ for $\Fi=p$ and $\Fi=\BP p$:

\begin{center}
\begin{tabular}{l}
$\D\BP p$ \ $\tooo{\mathsf{(A1)}}$ \ $\D\B p\hphantom{\BP}$ \ $\tooo{.2}$ \ $\B\D p$. \quad  $(a)$
\\
$\D\BP p$ \ $\tooo{\mathsf{(A2)}}$ \ $\D\B\BP p$ \ $\tooo{.2}$ \ $\B\D\BP p$. Hence:
\\
$\D\BP p$ \ $\tooo{(\mathsf{R\BP})}$ \ $\BP\D\BP p$ \ $\tooo{(a)}$ \ $\BP\B\D p$. \quad (b)
\\
$\D\BP p$ \ $\tooo{(\mathsf{A\BP})}$ \ $\BP\D p$, \quad obtained from $(a)$ and $(b)$.
\end{tabular}
\end{center}

\noindent
{\bf (1')} We obtain $\DP\B p\too\B\DP p$ by duality from (1).

\smallskip\noindent{\bf (2)} Derive $\DP\BP p\too\BP\DP p$ using (1') similarly
(replace $\D$ with~$\DP$ above).
\end{proof}

Note that the two stages of the derivation in the above lemma
correspond to two inductions needed to prove
that $R^+$ is convergent, assuming that $R$ is convergent.
First, by induction on $m$, one proves:
\begin{center}
$(x\,R^m\,y$ and $x\,R\,z)$ \ $\To$ \ $\exists t$: $(y\,R\,t$ and $z\,R^m\,t)$.
\end{center}
Secondly, by induction on~$n$ one proves:
\begin{center}
$(x\,R^m\,y$ and $x\,R^n\,z)$ \ $\To$ \ $\exists t$: $(y\,R^n\,t$ and $z\,R^m\,t)$.
\end{center}
Now, if $x\,R^+\,y$ and $x\,R^+\,z$, then 
$x\,R^m\,y$ and $x\,R^n\,z$, for some $m,n$.
Then there is $t$ such that $y\,R^n\,t$ and $z\,R^m\,t$.
Hence $y\,R^+\,t$ and $z\,R^+\,t$. So, $R^+$ is convergent.

This additionally justifies the name `induction axiom' for the axiom~$\mathsf{(A3)}$.

\subsection{On the semantics of Segerberg's axioms}

\begin{lemma}\label{Lem:Segerberg:Axioms:Semantics}
Let $F=(W,R,S)$ be a bi-modal frame.

\noindent
\begin{tabular}{@{}l@{\ \ }lcl@{}}
{\sf (1)} & $F\valid (A1)$ & $\IFF$ & $S\su R$. \\
{\sf (2)} & $F\valid (A2)$ & $\IFF$ & $S\su R\circ S$. \\
{\sf (3)} & $F\valid (A1)\land (A2)$ & $\TO$ & $S\su R^+$; the converse does not hold in general. \\
{\sf (4)} & $F\valid (A3)$ & $\TO$ & $S\se R^+$; the converse does not hold in general. \\
{\sf (5)} & \multicolumn{3}{@{}l}{$F\valid (A1)\land (A2)\land (A3)$ \qquad $\IFF$ \qquad $S=R^+$.}
\end{tabular}
\end{lemma}
\begin{proof}
This is a known fact.
\end{proof}

\subsection{On induction axiom and minimal filtrated frame}

Let us strengthen Lemma~\ref{Lem:A3:min:filtration} (recall that ${G\valid(A3)}$ implies ${S\se R^+}$).
Denote the \emph{minimal filtered} (through~$\Phi$) frame by
$G^{\mathsf{min}}_{\sim_\Phi}=(W/{\sim_\Phi},R^{\mathsf{min}}_{\sim_\Phi},S^{\mathsf{min}}_{\sim_\Phi})$.

\begin{lemma}[Induction axiom and minimal filtrated frame]\label{Lem:A3:min:frame} \ \\
%
Let $M=(W,R,S,V)\models(A3)^*$ and let $\Phi\se\Fm$ be finite.
Then ${G^{\mathsf{min}}_{\sim_\Phi}\valid(A3)}$.
\end{lemma}
\begin{proof}
Denote 
$\fM:=M^{\mathsf{min}}_{\sim_\Phi}=(G^{\mathsf{min}}_{\sim_\Phi},\fV)$,
where $G^{\mathsf{min}}_{\sim_\Phi}=(\fW,R^{\mathsf{min}}_{\sim_\Phi},S^{\mathsf{min}}_{\sim_\Phi})$.
Note that $\fM$ is a $\Phi$-filtration of~$M$,
since $R^{\mathsf{min}}_{\sim_\Phi}\se R^{\mathsf{max}}_{\Phi}$ and similarly for~$S$.
Therefore, $\fM$ is a finite \emph{differentiated} model:
indeed, if ${\fx\ne\fy}$, then $(M,x)\not\sim_\Phi(M,y)$, 
hence $(\fM,\fx)\not\sim_\Phi(\fM,\fy)$, by the Filtration lemma~\ref{Lemma:Filtration}.
Due to Lemma~\ref{Lem:Fin:Diff:Model},
%
in order to prove our lemma, it suffices to show that
\begin{center}
${M\models(A3)^*}$ \ implies \ ${\fM:=M^{\mathsf{min}}_{\sim_\Phi}\models(A3)^*}$.
\end{center}

Assume $\fM\nmodels(A3)[p:=B]$, for some formula~$B$.
Then there is ${\fx\in\fW}$ such that
{\bf(a)} $\fx\models\BP({B\to\B B})$,
{\bf(b)} ${\fx\models\B B}$,
{\bf(c)} ${\fx\nmodels\BP B}$.
Hence there is ${\fy\in\fW}$ such that $\fx\,\fS\,\fy$ and
{\bf(d)} ${\fy\nmodels B}$.
Since $\fx\,S^{\mathsf{min}}_{\sim_\Phi}\,\fy$, without loss of generality, $x\,S\,y$.

Consider  $Y:=\fV(B)=\{\fz\in\fW\mid\fM,\fz\models B\}$.
As in Lemma~\ref{Lem:A3:min:filtration},
$Y$ is a finite collection of definable
subsets of~$W$, hence their union ${\bigcup}Y$ is also a definable subset of~$W$.
So, there is a formula~$\Fi$ such that, for all ${z\in W}$, we have:
\begin{center}
$M,z\models\Fi$ \ $\Iff$ \ $z\in{\textstyle\bigcup}Y$ \ $\Iff$ \ 
$\fz\in Y=\fV(B)$ \ $\Iff$ \ $\fM,\fz\models B$.
\end{center}

Now let us show that $M,x\nmodels(A3)[p:=\Fi]$, in contradition with ${M\models(A3)^*}$.

\noindent
{\bf(a')} $M,x\models\BP({\Fi\to\B\Fi})$. Indeed, take any $a,b\in W$ such that
$x\,S\,a\,R\,b$ and ${a\models\Fi}$. Then $\fx\,\fS\,\fa\,\fR\,\fb$ and ${\fa\models B}$.
Hence ${\fb\models B}$ by~{\bf(a)}, and so ${b\models\Fi}$.

\noindent
{\bf(b')} ${M,x\models\B\Fi}$. Indeed, if $x\,R\,z$, then $\fx\,\fR\,\fz$;
hence ${\fz\models B}$ by~{\bf(b)}, so ${z\models\Fi}$.

\noindent
{\bf(d')} $M,x\nmodels\BP\Fi$. Indeed, $x\,S\,y$ and $M,y\nmodels\Fi$, because $\fy\nmodels B$ by~{\bf(d)}.
\end{proof}

{\bf Question:} What other formulas transfer from $M$ to $\fM$ and back?
Maybe, even: $M\models A[p:=\Fi]$ iff $\fM\models A[p:=B]$, for all formulas $A$ with $\Var(A)=\{p\}$?

\medskip
Another question is: {\it What other modal formulas $\Fi$ have the property 
from the above lemma, namely:
\begin{center}
if $M\models \Fi^*$, then $\fF^{\mathsf{min}}_{\sim_\Phi}\valid \Fi$,
\end{center}
for any finite set of modal formulas~$\Phi$?
}

%% file: _AiML2020_Main.bbl
\begin{thebibliography}{10}
\expandafter\ifx\csname url\endcsname\relax
  \def\url#1{\texttt{#1}}\fi
\expandafter\ifx\csname urlprefix\endcsname\relax\def\urlprefix{URL }\fi
\newcommand{\enquote}[1]{``#1''}

\bibitem{B:R:V:ML:2002}
Blackburn, P., M.~de~Rijke and Y.~Venema, \enquote{Modal Logic,}  Cambridge
  Tracts in Theoretical Computer Science  \textbf{53}, Cambridge University
  Press, 2002.

\bibitem{Ch:Za:ML:1997}
Chagrov, A. and M.~Zakharyaschev, \enquote{Modal Logic,}  Oxford Logic Guides
  \textbf{35}, Oxford University Press, 1997.

\bibitem{dunin2011teamwork}
Dunin-Keplicz, B. and R.~Verbrugge, \enquote{Teamwork in multi-agent systems: A
  formal approach,}  Wiley Series in Agent Technology  \textbf{21}, John Wiley
  \& Sons, 2011.

\bibitem{fagin2003reasoning}
Fagin, R., Y.~Moses, J.~Y. Halpern and M.~Y. Vardi, \enquote{Reasoning about
  knowledge,} MIT press, 2003.

\bibitem{fischer1979propositional}
Fischer, M.~J. and R.~E. Ladner, \emph{Propositional dynamic logic of regular
  programs}, Journal of computer and system sciences \textbf{18} (1979),
  pp.~194--211.

\bibitem{Libkin14}
Gheerbrant, A., L.~Libkin and C.~Sirangelo, \emph{Na{\"{\i}}ve evaluation of
  queries over incomplete databases}, {ACM} Trans. Database Syst. \textbf{39}
  (2014), pp.~31:1--31:42.

\bibitem{GoldbLTC}
Goldblatt, R., \enquote{Logics of Time and Computation,} Number~7 in CSLI
  Lecture Notes, Center for the Study of Language and Information, 1992, 2nd
  edition.

\bibitem{itas2015}
Kikot, S., \emph{First-order formulas that are preserved under minimal
  filtration}, in: \emph{Proceedings of 39th International Workshop of IITP RAS
  ``Information Technologies and Systems 2015''}, 2015, pp. 635--639, (in
  {R}ussian).

\bibitem{KSZ:AiML:2014}
Kikot, S., I.~Shapirovsky and E.~Zolin, \emph{Filtration safe operations on
  frames}, in: R.~Gor{\'{e}}, B.~P. Kooi and A.~Kurucz, editors, \emph{Advances
  in Modal Logic}, 10 (2014), pp. 333--352.

\bibitem{KozenParikh1981}
Kozen, D. and R.~Parikh, \emph{An elementary proof of the completeness of
  {PDL}}, Theoretical Computer Science \textbf{14} (1981), pp.~113--118.

\bibitem{WolterKracht1991}
Kracht, M. and F.~Wolter, \emph{Properties of independently axiomatizable
  bimodal logics}, J. Symbolic Logic \textbf{56} (1991), pp.~1469--1485.

\bibitem{Segerberg1968}
Segerberg, K., \emph{Decidability of four modal logics}, Theoria \textbf{34}
  (1968), pp.~21--25.

\bibitem{Segerberg1977}
Segerberg, K., \emph{A completeness theorem in the modal logic of programs},
  Notices of the American Mathematical Society \textbf{A-522} (1977),
  pp.~77T--E69.

\bibitem{Segerberg1982}
Segerberg, K., \emph{A completeness theorem in the modal logic of programs},
  Banach Center Publications \textbf{9} (1982), pp.~31--46.

\bibitem{Segerberg:1993:PDL}
Segerberg, K., \enquote{A Concise Introduction to Propositional Dynamic Logic,}
  1993.

\bibitem{Shehtman1987}
Shehtman, V., \emph{On some two-dimensional modal logics}, in: \emph{8th
  Congress on Logic Methodology and Philosophy of Science, abstracts}, volume 1
  (1987), pp. 326--330.

\bibitem{Shehtman:AiML:2004}
Shehtman, V., \emph{Filtration via bisimulation}, in: R.~Schmidt,
  I.~Pratt-Hartmann, M.~Reynolds and H.~Wansing, editors, \emph{Advances in
  Modal Logic}, 5 (2004), pp. 289--308.

\bibitem{Shehtman:AiML:2014}
Shehtman, V., \emph{Canonical filtrations and local tabularity}, in:
  R.~Gor{\'{e}}, B.~P. Kooi and A.~Kurucz, editors, \emph{Advances in Modal
  Logic}, 10 (2014), pp. 498--512.

\end{thebibliography}
